\documentclass{amsart} 
\usepackage{amsmath,amsthm}
\usepackage{amssymb,amscd,latexsym}
\usepackage{boxedminipage}
\usepackage{eepic}
\usepackage{epic}
\usepackage{wrapfig}
\usepackage[dvips]{graphicx}

\newcommand{\erase}[1]{}
\newtheorem{theorem}{Theorem}[section]

\newtheorem{proposition}[theorem]{Proposition}

\newtheorem{corollary}[theorem]{Corollary}
\newtheorem{remark}[theorem]{Remark}
\newtheorem{fact}[theorem]{Fact}
\newtheorem{lem}[theorem]{Lemma}
\numberwithin{equation}{section}
\newcommand{\bp}{\begin{pmatrix}}
\newcommand{\ep}{\end{pmatrix}}
\newcommand{\bps}{\begin{smallmatrix}}
\newcommand{\eps}{\end{smallmatrix}}

\def\R{{\mathbb R}}
\def\Z{{\mathbb Z}}

\def \0{{\bf 0}}
\def \1{{\bf 1}}

\def \rank{\mathrm{rank}}

\def \mf#1#2#3#4{
\xymatrix{{#1}\  \ar@<0.4ex>[r]^{{#2}} & \ {#4}
\ar@<0.4ex>[l]^{{#3}}
}
}

\def \mfs#1#2#3#4{\!
\xymatrix@C=1.5em{{#1} \! \ar@<0.2ex>[r]^{{#2}} & \! {#4}
\ar@<0.2ex>[l]^{{#3}}
}
\!}

\def \mfl#1#2#3#4{
\xymatrix@C=2.6em{{#1}\  \ar@<0.4ex>[r]^{{#2}} &\  {#4}
\ar@<0.2ex>[l]^{{#3}}
}
}

\def \mfss#1#2#3#4{\!
\xymatrix@C=1.5em{{#1} \ar@<0.3ex>[r]^{{#2}} & {#4}
\ar@<0.3ex>[l]^{{#3}}
}
\!}

\begin{document}
\title{On relations for zeros of $f$-polynomials and $f^{+}$-polynomials.} 

\author{Tadashi Ishibe}

%\hmjlogo{}{}{}{}
%\address{Department of Mathematical Sciences, \\
%University of Tokyo, \\
%3-8-1 Komaba Meguro-ku Tokyo, 153-8914 Japan
%}
  
%\date{}

%\vspace{ -0.6cm} Editor

% {\renewcommand{\baselinestretch}{0.9} Editor

%\email{tishibe@ms.u-tokyo.ac.jp}
%  \vspace{0.3cm}

%\subjclass{20F05}

%\keywords{monoid, fundamental group, the word problem, the conjugacy problem}

%2000 Mathematics Subject Classification Numbers:  20F05 \\
%key words and phrases:  monoid, fundamental group, the word problem, the conjugacy problem

\maketitle
\begin{abstract}
Let $\Phi$ be an irreducible (possibly noncrystallographic) root system of rank $l$ of type $P$. For the corresponding cluster complex $\Delta(P)$, which is known as pure $(l-1)$-dimensional simplicial complex, we define the generating function of the number of faces of $\Delta(P)$ with dimension $i-1$, which is called the {\it $f$-polynomial}. We show that the $f$-polynomial has exactly $l$ simple real zeros on the interval $(0, 1)$ and the smallest root for the infinite series of type $A_l$, $B_l$ and $D_l$ monotone decreasingly converges to zero as the rank $l$ tends to infinity. We also consider the generating function (called the {\it $f^{+}$-polynomial}) of the number of faces of the positive part $\Delta_{+}(P)$ of the complex $\Delta(P)$ with dimension $i-1$, whose zeros are real and simple and are located in the interval $(0, 1]$, including a simple root at $t=1$. We show that the roots $\{ t^{+}_{P, \nu+1} \}_{\nu=1}^{l-1}$ in decreasing order of $f^{+}$-polynomial alternate with the roots $\{ t_{P, \nu} \}_{\nu=1}^{l}$ in decreasing order of $f$-polynomial.
\end{abstract}
\section{Introduction}
Let $\Phi$ be an irreducible (possibly noncrystallographic) root system of rank $l$ of type  $P\in${\small $\{A_l\ (l\!\ge\!1), B_l\ (l\! \ge\!2), D_l\ (l\! \ge\! 4), E_l\ (l=6, 7, 8), F_4, G_2, H_3, H_4, I_2(p)\,(p\ge3)\}$}. Let $\Phi^{+}$ be a positive system for $\Phi$ with corresponding simple system $\Pi$. The cluster complex $\Delta(P)$\footnote{In \cite{[F-Z]}, the notion of cluster complex $\Delta(P)$ was originally defined for the case where $\Phi$ is crystallographic. As was noted in \cite{[F-R2]}, the constructions in \cite{[F-Z]} extend verbatim to the noncrystallographic types.} introduced by Fomin-Zelevinsky (\cite{[F-Z]}) is a pure $(l-1)$-dimensional simplicial complex whose ground set is the set $\Phi_{\ge{-1}}:=\Phi^{+}\sqcup(-\Pi)$ of almost-positive roots and its geometric realization is homeomorphic to a sphere. Let $\Delta_{+}(P)$ denote the induced subcomplex of $\Delta(P)$ on the vertex set $\Phi^{+}$. The complex $\Delta_{+}(P)$ is referred to as the {\it positive part} of $\Delta(P)$. For a non-negative integer $i$ with $0\le i \le d-1$, let $f_{i}$ (resp. $f^{+}_{i}$) denote the number of faces of $\Delta(P)$ (resp. $\Delta_{+}(P)$) with dimension $i-1$. We call the sequence $(f_{0}, f_{1}, \ldots, f_{l-1})$ (resp. $(f^{+}_{0}, f^{+}_{1}, \ldots, f^{+}_{l-1})$) the $f$-vector of the complex $\Delta(P)$ (resp. $\Delta_{+}(P)$). Then, we define the $f$-polynomial and $f^{+}$-polynomial of type $P$ in the formal variable $t$ by
{\footnotesize\[
 f_P(t) := \sum_{i=0}^{l}f_{i-1}(P)(-t)^{i},\,\, f^{+}_P(t) := \sum_{i=0}^{l}f^+_{i-1}(P)(-t)^{i},
\]}
where {\small$f_{-1}(P):= 1$} and {\small$f^{+}_{-1}(P):= 1$}.
In this article, we will study the zero loci of the $f$-polynomial. In a similar manner to \cite{[I-S]}, we will show that the $f$-polynomial has an unexpected strong connection with orthogonal polynomials. By making the best use of these properties, we will show that the $f$-polynomial has exactly $l$ simple real zeros on the interval $[0, 1]$ and the smallest root for the infinite series of type $A_l$, $B_l$ and $D_l$ monotone decreasingly converges to zero as the rank $l$ tends to infinity. Furthermore, we obtain inequalities for zeros of $f$-polynomials and $f^{+}$-polynomials.\par
In \cite{[I-S]}, the authors studied the zero loci of the {\it skew-growth function} (\cite{[Sa2]})\footnote
{For a positive homogeneously finitely presented monoid $M ={\langle L \mid R\,\rangle}_{mo}$ that satisfies the cancellation condition (i.e. $axb = ayb$ implies $x = y$ ) and is equipped with a degree map $\deg\!:\!M \to\!\Z_{\ge0}$ defined by assigning to each equivalence class of words the length of the words, the {\it spherical growth function} for the monoid $M$ is defined as
\[
 P_{M, \deg}(t)\!:=\!\!\sum_{u\!\in\! M}\!t^{\deg(u)}.
\] 
If any two elements in the monoid $M$ admit left (resp. right) least common multiple, the {\it skew-growth function} for it is defined as
\[
 N_{M, \deg}(t)\!:=\!\!\sum_{J \subset I_{0}}\!(-1)^{\#J}t^{\deg(\mathrm{lcm}_{r}(J))},
\] 
where $I_{0}:=\mathrm{min}_{l}(M \setminus \{ 1 \})$ (the set of minimal elements of $M$ with respect to the partial ordering induced by the left division relation). Then, the following inversion formula for $M$ holds
\[
P_{M, \deg}(t)\cdot N_{M, \deg}(t) = 1.
\]
This fact is generalized in \cite{[Sa2]}. The distribution of the zeros of the skew-growth functions are investigated from the viewpoint of limit functions (\cite{[Sa3]}).}
of a {\it dual Artin monoid} of finite type (\cite{[Be]}). They noted that the skew-growth function for the dual
\newpage
\noindent
 Artin monoid of type $P$ coincides with the $f^{+}$-polynomial (for the complex $\Delta_{+}(P)$) of the same type\footnote{This fact is proved in \cite{[At]}Corollary 2.5. In addition to this fact, it is known that the skew-growth function is identified with the generating function of M\"obius invariants (called the {\it characteristic polynomial}) of the {\it lattice of non-crossing partitions} (\cite{[K], [Be], [B-W]}).}. Suggested by some numerical experiments concerning the $f^{+}$-polynomials, they conjectured the following 1, 2 and 3 (\cite{[I-S]})\footnote{In \cite{[Sa1]}, the author also gave three conjectures on the skew-growth function for an Artin monoid of finite type.}. \\
\noindent
1.\ {\small$f^{+}_P(t)/(1\!-\!t)$} is an irreducible polynomial over $\Z$, up to the trivial factor $1-2t$ for the types $A_l$ ($l$: even) and $D_4$. 

\noindent
2.\ {\small$f^{+}_P(t)$} has  $l=\rank(P)$ simple real roots on the interval $(0, 1]$, including a simple root at $t=1$.

\noindent
3.\ The smallest root of {\small$f^{+}_P(t)$} monotonously decreasingly converges to $0$ as the rank $l$ tends to infinity for the infinite series of type $A_l$, $B_l$ and $D_l$.  \\

\smallskip
By analogy with 1, 2 and 3, and also inspired by some numerical experiments (see Appendix II for the figures of the zero loci of the functions of types $A_{20},\ B_{20},\ D_{20}$ and $E_8$), we conjecture the following. 

\medskip
\noindent
{\bf Conjecture 1.}  $f_P(t)$ is an irreducible polynomial over $\Z$, up to the trivial factor $1-2t$ for the types $A_l$ ($l$: odd), $B_l$ ($l$: odd), $D_l$ ($l$: odd),  $E_7$ and $H_3$.

\smallskip
\noindent
{\bf Conjecture 2.}  $f_P(t)$ has  $l=\rank(P)$ simple real roots on the interval $(0, 1)$.

\smallskip
\noindent
{\bf Conjecture 3.}  The smallest root of $f_P(t)$ monotonously decreasingly converges to $0$ as the rank $l$ tends to infinity for the infinite series of type $A_l$, $B_l$ and $D_l$. 
\begin{remark}
{\rm Conjecture 1. is approved  for types $A_l$   ($1\le l \le 30$), 
$B_l$  ($2\le l \le 30$),  $D_l$   ($4\le l \le 30$), $E_6$, $E_7$, $E_8$, $F_4$, $G_2$, $H_3$, $H_4$ and $I_2(p)$ ($p\ge 3$)  by using the software package Mathematica on the Table A and B in Appendix I.  }
\end{remark}
In addition to the above three conjectures, we give one more conjecture on a relation for zeros of $f$-polynomials and $f^{+}$-polynomials. For $P\in${\small $\{A_l\ (l\!\ge\!1), B_l\ (l\! \ge\!1), D_l\ (l\! \ge\! 4), E_l\ (l=6, 7, 8), F_4, G_2, H_3,$ $H_4, I_2(p)\ (p\ge3)\}$}, let $t_{P, \nu}$ be the zeros of $f_{P}(t)$ in decreasing order (i.e. $1 > t_{P, 1} > t_{P, 2}  > \cdots > t_{P, l} > 0$) and let $t^{+}_{P, \nu}$ be the zeros of $f^{+}_{P}(t)$ in decreasing order (i.e. $1 = t^{+}_{P, 1} > t^{+}_{P, 2}  > \cdots > t^{+}_{P, l} > 0$).\\
\smallskip
\noindent
{\bf Conjecture 4.}  The system $\{ t^{+}_{P, \nu+1} \}_{\nu=1}^{l-1}$ alternates with the system $\{ t_{P, \nu} \}_{\nu=1}^{l}$, that is,
\[
t_{P, \nu} > t^{+}_{P, \nu+1} >t_{P, \nu+1} , \,\,\,\,(\nu= 1, \ldots, l-1).
\]
\par

%and we define
%{\footnotesize\[
% f^+_P(t) := \sum_{i=0}^{l}f^+_{i-1}(P)(-t)^{i}, h^+_P(t) := \sum_{i=0}^{l}h^+_%{i}(P)t^{i}. 
%\]}
%{\footnotesize\[
% \sum_{i=0}^{l}f^+_{i-1}(P)(x-1)^{l-i}=  \sum_{i=0}^{l}h^+_{i}(P)x^{l-i}. 
%\]}
The aim of the present paper is to give affirmative answers to Conjectures 2, 3 and 4. In \S2, we first prepare some useful functional equations (2.1), (2.2) and (2.3) for later use. In the functional equation for type $A_l$ (resp. $B_l$), the $f$-polynomial of type $A_l$ (resp. $B_l$) has a simple relation with the $f^{+}$-polynomial of the same type. Although for type $D_l$ the $f$-polynomial does not have a simple relation with the $f^{+}$-polynomial of the same type, the $f$-polynomial of type $D_l$ has a rather complicated relation with the $f$-polynomial of type $B_l$. Hence, from the equation (2.2) and the equation (6.4) in \cite{[I-S]}, we may say that the $f$-polynomial of type $D_l$ has a certain relation with $f^{+}$-polynomial of type $D_l$. Secondly, for the three infinite series $A_l,\ B_l$ and $D_l$ of $f$-polynomials, we show that they are expressed in higher (logarithmic) derivatives of the polynomials of the form $t^{p}(1-t)^{q}$. In analogy with the classical Rodrigues's formula in the theory of orthogonal polynomials \cite{[Sz]}, we call the formula {\it Rodrigues type formula}. As a consequence of the formulae, the polynomials are expressed by a use of orthogonal polynomials. In \S3, we will show that the series of the $f$-polynomials of type $A_l$ and $B_l$ satisfies $3$-term recurrence relation, respectively, and the series of the $f$-polynomials of type $D_l$ satisfies $4$-term recurrence relation. In \S4, we will prove Conjecture 2 except for types $D_l$ affirmatively. For type $A_l$ (resp. $B_l$), the $f$-polynomial, up to a constant factor, coincides with a certain orthogonal polynomial. Therefore, Conjecture 2 for types $A_l$ and $B_l$ is true (\cite{[Sz]}). For exceptional types $E_6,\ E_7,\ E_8,\ F_4,\ G_2$ and non-crystallographic types $H_3,\  H_4$ and $I_2(p)$, we will construct Sturm's sequence on the interval $[0, 1]$. Due to the Sturm Theorem (see for instance \cite{[T]} Theorem 3.3), Conjecture 2 for them is true. In \S5, we will prove Conjecture 2 for types $D_l$ affirmatively. By applying the intermediate value theorem to the equation (2.3), we show that Conjecture 2 for types $D_l$ is true. In \S6, we will prove Conjecture 3 affirmatively. For series $A_l$, due to the equation (2.1), the smallest zero locus of {\small$f_{A_l}(t)$} coincides with the smallest zero locus of {\small$f^{+}_{A_{l+1}}(t)$}. From Theorem 6.1 in \cite{[I-S]}, we show that Conjecture 3 for series $A_l$ is true. For series $B_l$, from the equation (2.9), we show that the smallest zero locus of {\small$f_{B_l}(t)$} coincides with the smallest zero locus of the shifted Legendre polynomial {\small$\widetilde{P}_{l}^{(0, 0)}(t):=P_l^{(0, 0)}(2t-1)$} (the shifting of Legendre polynomial {\small$P_l^{(0, 0)}(t)$}) of degree $l$. From \cite{[Sz]} Theorem 6.21.3, we have that the smallest zero locus of Legendre polynomial {\small$P_{l}^{(0,0)}(t)$} monotonously decreasingly converges to zero as the rank $l$ tends to infinity. Therefore, we show that Conjecture 3 for series $B_l$ is true. The proof for the series $D_l$ uses again the equation (2.3), where the polynomials of type  $D_l$ are expressed by those of type $B_l$ so that the roots of type $D_l$ are sandwiched by the roots of type $B_l$. Hence, we show that Conjecture 3 for types $D_l$ is true. In \S7, we will prove Conjecture 4 except for types $D_l$ affirmatively. For series $A_l$, due to the equation (2.1), we have that Conjecture 4 for series $A_l$ is true. Since the $f$-polynomial of type $B_l$, up to a constant factor, coincides with the shifted Legendre polynomial {\small$\widetilde{P}_{l}^{(0, 0)}(t)$}, due to Proposition 6.6 in \cite{[I-S]}, we have that Conjecture 4 for series $B_l$ is true. 
In \S8, we will prove Conjecture 4 for types $D_l$ affirmatively. The proof is divided into two parts and is more complicated. In Part I, we discuss location of the following roots $t_{D_{l}, l+1-\nu}$ and $t^{+}_{D_{l}, l+1-\nu}$ ($\nu= 1, \ldots, \lfloor l/2 \rfloor$). In Part II, we discuss location of the following roots $t_{D_{l},  \lceil l/2 \rceil+1-\nu}$ and $t^{+}_{D_{l}, \lceil l/2 \rceil+1-\nu}$ ($\nu= 1, \ldots, \lceil l/2 \rceil$).

\begin{remark} 
In \cite{[Br]}, the author showed that the $h$-polynomial of type $P$ has $l$ simple real roots on the interval $(-\infty, 0)$. The $h$-polynomial of type $P$ in the formal variable $t$ is defined as
{\footnotesize\[
 h_P(t) := \sum_{i=0}^{l}h_{i}(P)t^{i},
\]}
where the coefficients $h_{i}(P)$ are defined by the equation
{\footnotesize\[
 \sum_{i=0}^{l}f_{i-1}(P)(x-1)^{l-i}=  \sum_{i=0}^{l}h_{i}(P)x^{l-i}.
\]}
Hence, we have
{\footnotesize\begin{equation}
 f_P(t) = (1-t)^{l}h_P(\frac{t}{t-1}).
\end{equation}}
From (1.1), this implies that Conjecture 2. is approved.
\end{remark}
\section{Rodrigues type formulae and orthogonal polynomials}
In this section, we first prepare some useful propositions for later use. Next, for the three infinite series $A_l,\ B_l$ and $D_l$ of $f$-polynomials, we show that they are expressed in higher (logarithmic) derivatives of the polynomials of the form $t^{p}(1-t)^{q}$. In analogy with the classical Rodrigues's formula \cite{[Sz]}, we call the formula {\it Rodrigues type formula}. As a consequence of the formulae, the polynomials are expressed by a use of orthogonal polynomials.
\begin{proposition}
1.  For type $A_l$, the following identity holds for $l = 1, 2, \ldots$:
\begin{equation}
(1-t)f_{A_{l}}(t) = f^+_{A_{l+1}}(t).
\end{equation}
2.  For type $B_l$, the following identity holds for $l = 2, 3, \ldots$:
\begin{equation}
f_{B_{l}}(t) + f_{B_{l-1}}(t) = 2f^+_{B_{l}}(t).
\end{equation}
3.  The following identity holds for $l = 4, 5, \ldots$:
\begin{equation}
f_{D_{l}}(t) = \frac{l-2}{2(l-1)}f_{B_{l}}(t) + \frac{l}{2(l-1)}(1-2t)f_{B_{l-1}}(t).
\end{equation}
\end{proposition}
\begin{proof}
1: From Table A, we compute the coefficient of $(-t)^k$ of $(1-t)f_{A_l}(t)$.
{\footnotesize\[
\frac{1}{l+2}\binom {l}{k}\binom {l+k+2}{k+1} + \frac{1}{l+2}\binom {l}{k-1}\binom {l+k+1}{k} = \frac{1}{l+1}\binom {l+1}{k}\binom {l+k+1}{k+1}.
\]}
This coincides with the coefficient of $(-t)^k$ of $N_{G^{dual +}_{A_{l+1}}}(t)$ in Table A in \cite{[I-S]}.\\
2: From Table A, we compute the coefficient of $(-t)^k$ on the LHS of (2.2).
{\footnotesize\[
\binom {l}{k}\binom {l+k}{k} + \binom {l-1}{k}\binom {l+k-1}{k} = 2\binom {l}{k}\binom {l+k-1}{k}.
\]}
This coincides with the coefficient of $(-t)^k$ of $2f^{+}_{B_l}(t) (= 2N_{G^{dual +}_{B_{l}}}(t))$ in Table A in \cite{[I-S]}.\\
3: From Table A, we compute the coefficient of $(-t)^k$ on the RHS of (2.3).
{\footnotesize\[
\frac{l-2}{2(l-1)}\cdot\frac{(l+k)!}{(l-k)!k!k!} + \frac{l}{2(l-1)}\cdot\frac{(l+k-1)!}{(l-1-k)!k!k!} + \frac{l(l+k-2)!}{(l-1)(l-k)!(k-1)!(k-1)!}
\]}
{\footnotesize\[
= \frac{(l+k-2)!}{(l-k)!k!k!}\biggl\{  l(l+k-1) + k(k-1) \biggr\}.\,\,\,\,\,\,\,\,\,\,\,\,\,\,\,\,\,\,\,\,\,\,\,\,\,\,\,\,\,\,\,\,\,\,\,\,\,\,\,\,\,\,\,\,\,\,\,\,\,\,\,\,\,\,\,\,\,\,\,\,\,\,\,\,\,\,\,\,\,\,\,\,\,\,\,\,\,\,\,\,\,\,\,\,\,\,\,\,\,\,\,\,\,\,\,\,\,\,\,\,\,\,\,\,\,\,\,
\]}
This coincides with the coefficient of $(-t)^k$ on the left hand side in Table A.
\end{proof}

\medskip
\begin{theorem} {\bf (Rodrigues type formula) }
\label{Rodrigues} 
 For types $A_l \ (l\ge1)$, $B_l\ (l\ge2)$ and $D_l \ (l\ge4)$, we have the formulae:
{\small
\begin{align}
\label{RodriguesA}
t(1-t)f_{A_l}(t) &= \frac{1}{(l+1)!} \frac{\mathrm{d}^{l}}{\mathrm{d}t^{l}}\biggl[ t^{l+1}(1-t)^{l+1} \biggr], \\
\label{RodriguesB}
f_{B_{l}}(t) & =  \frac{1}{l!} \frac{\mathrm{d}^{l}}{\mathrm{d}t^{l}}\biggl[ t^{l}(1-t)^l \biggr], \\
\label{RodriguesD}
f_{D_{l}}(t)  &= \frac{1}{(l-1)!}\frac{\mathrm{d}^{l-1}}{\mathrm{d}t^{l-1}}\biggl[ t^{l-1}(1-t)^{l} \biggr] + \frac{1}{(l-2)!}\frac{\mathrm{d}^{l-2}}{\mathrm{d}t^{l-2}}\biggl[ t^{l}(1-t)^{l-2} \biggr]\\
&= \frac{1}{(l-1)!}\frac{\mathrm{d}^{l-2}}{\mathrm{d}t^{l-2}}\biggl[ t^{l-2}(1-t)^{l-2} \Bigl\{(l-1) - (3l-2)t + (3l-2)t^2 \Bigr\} \biggr].
\end{align}
}
\end{theorem}
\begin{proof} 
Type $A_l$: Due to Theorem 2.1 in \cite{[I-S]}, the right hand side of (2.4) coincides with $tf^{+}_{A_{l+1}}(t)$. Thanks to (2.1), this coincides with the left hand side of (2.4).\\

\smallskip
\noindent
Type $B_l$: The right hand side of (\ref{RodriguesB}) is calculated as 
{\footnotesize
\[
\frac{1}{l!} \frac{\mathrm{d}^{l}}{\mathrm{d}t^{l}}\biggl[ t^{l}(1-t)^l\biggr] = 
\frac{1}{l!} \frac{\mathrm{d}^{l}}{\mathrm{d}t^{l}}\biggl[\sum_{k=0}^{l} (-1)^{k} \binom {l}{k} t^{l+k}\biggr]= \sum_{k=0}^{l}(-1)^{k}\frac{(l+k)!}{(l-k)!k!k!}t^k.
\]}
This gives RHS of the expression of $f_{B_{l}}(t)$ in Table A.\\

\smallskip
\noindent
Type $D_l$: We compute the right hand side of (\ref{RodriguesD}).
% {\scriptsize\[
%  \frac{1}{(l-2)!}\frac{\mathrm{d}^{l-2}}{\mathrm{d}t^{l-2}}\biggl[ t^{l-2}(1-t)^{l} \biggr] + \frac{1}{(l-3)!}\frac{\mathrm{d}^{l-3}}{\mathrm{d}t^{l-3}}\biggl[ t^{l-1}(1-t)^{l-2} \biggr]\,\,\,\,\,\,\,\,\,\,\,\,\,\,\,\,\,\,\,\,\,\,\,\,\,\,\,\,\,\,\,\,\,\,\,\,\,\,\,\,\,\,\,\,\,\,\,\,\,\,\,\,\,\,
% \]}
{\footnotesize\[
\ \  \frac{1}{(l-1)!} \frac{\mathrm{d}^{l-1}}{\mathrm{d}t^{l-1}}\biggl[\sum_{k=0}^{l} (-1)^{k} \binom {l}{k} t^{l+k-1} \biggr] + \frac{1}{(l-2)!}\frac{\mathrm{d}^{l-2}}{\mathrm{d}t^{l-2}}\biggl[\sum_{k=0}^{l-2} (-1)^{k} \binom {l-2}{k} t^{l+k} \biggr]
 \]}
\vspace{-0.1cm}
{\footnotesize\[
= \sum_{k=0}^{l}(-1)^{k}\frac{l(l+k-1)!}{(l-k)!k!k!}t^k + \sum_{k=0}^{l-2}(-1)^{k}\frac{(l+k)!}{(l-2-k)!k!(k+2)!}t^{k+2} \,\,\,\,\,\,\,\,\,\,\,\,\,\,\,\,\,\,\,\,\,\,\,\,\,\,\,
\]}
\vspace{-0.1cm}
{\footnotesize\[
= \sum_{k=0}^{l}(-1)^{k}\frac{l(l+k-1)!}{(l-k)!k!k!}t^k + \sum_{k=2}^{l}(-1)^{k}\frac{(l+k-2)!}{(l-k)!k!(k-2)!}t^{k} \,\,\,\,\,\,\,\,\,\,\,\,\,\,\,\,\,\,\,\,\,\,\,\,\,\,\,\,\,\,\,\,\,\,\,\,\,\,\,\,\,\,\,
\]}
\vspace{-0.1cm}
{\footnotesize\[
= \sum_{k=0}^{l} (-1)^{k} \biggl( \binom {l}{k}\binom {l+k-1}{k} + \binom {l-2}{k-2}\binom {l+k-2}{k}\biggr) t^k.\,\,\,\,\,\,\,\,\,\,\,\,\,\,\,\,\,\,\,\,\,\,\,\,\,\,\,\,\,\,\,\,\,\,\,\,\,\,\,\,\,\,\,\,\,\,\,\,\,
\vspace{-0.1cm}
\]}
This gives RHS of the expression of $f_{D_{l}}(t)$ in Table A.
\end{proof} 

For $l\in\Z_{\ge0}$ and $\alpha, \beta\in \R_{>-1}$, let $P^{(\alpha, \beta)}_{l}(x)$ be the Jacobi polynomial (c.f.\ \cite{[Sz]} 2.4). Let us introduce the  {\it shifted Jacobi polynomial} of degree $l$ by setting
\[
 \widetilde{P}^{(\alpha, \beta)}_{l}(t) := P^{(\alpha, \beta)}_{l}(2t-1) . %\quad \& \quad\widetilde{P}^{(\alpha, \beta)}_{0}(t) := 1.
\]
\begin{fact}\cite{[Sz]}(4.3.1) 
\label{Jacobi}
The shifted Jacobi polynomial %$\widetilde{P}^{(\alpha, \beta)}_{l}(t)$ 
satisfies the following equality
\[
\begin{array}{c}
(t-1)^{\alpha}t^{\beta}\widetilde{P}^{(\alpha, \beta)}_{l}(t) = \frac{1}{l!} \frac{\mathrm{d}^{l}}{\mathrm{d}t^{l}}\big[ (t-1)^{l + \alpha}t^{l + \beta} \big].
\end{array}
\]
\end{fact}
Comparing two formulae in Theorem \ref{Rodrigues} and Fact \ref{Jacobi}, we obtain expression of the $f$-polynomials for types $A_l$, $B_l$ and $D_l$ by shifted Jacobi polynomials.
\begin{align}
\label{JacobiA}
f_{A_{l}}(t) & \ =\  \frac{(-1)^{l}}{l+1}\widetilde{P}^{(1, 1)}_{l}(t),\\
\label{JacobiB}
f_{B_{l}}(t) & \ =\  (-1)^{l}\widetilde{P}^{(0, 0)}_{l}(t),\\
\label{JacobiD}
f_{D_{l}}(t) & \ = \ (-1)^{l-1}(1-t)\widetilde{P}^{(1, 0)}_{l-1}(t)+
(-1)^{l}t^2\widetilde{P}^{(0, 2)}_{l-2}(t).
\end{align}
\begin{remark}
From (2.3), we obtain the following expression of the $f$-polynomial for type $D_l$ by shifted Legendre polynomials.
\begin{equation}
(-1)^{l}f_{D_{l}}(t) = \frac{l-2}{2(l-1)}\widetilde{P}^{(0, 0)}_{l}(t) - \frac{l}{2(l-1)}(1-2t)\widetilde{P}^{(0, 0)}_{l-1}(t).
\end{equation}
\end{remark}
\section{Recurrence relations for types $A_l$ ($l\ge1$), $B_l$ ($l\ge2$) and $D_l$
($l\ge4$)}
As an application of the Rodrigues type formulae, we show that the the series of $f$-polynomials for types $A_l$ ($l\ge1$), $B_l$ ($l\ge2$) and $D_l$ satisfy either 3-term or 4-term recurrence relations (Theorem 3.1).
\medskip
\begin{theorem}   
For type $A_l$ and $B_l$, the following $3$-term recurrence relation holds. 
{\small
\begin{equation}
\label{recurrenceAB}
\begin{array}{rl}
\!\!\! (l+4)f_{A_{l+2}}(t) \!\!\! &= (2l+5)(1-2t)f_{A_{l+1}}(t) - (l+1)f_{A_{l}}(t),\\
\!\!\! (l+2)f_{B_{l+2}}(t) \!\!\! & =
 (2l+3)(1-2t)f_{B_{l+1}}(t) - (l+1)f_{B_l}(t).
\end{array}
\end{equation}  
}
For type $D_l$, the following $4$-term recurrence relation holds.
{\small
\begin{equation}
\label{recurrenceD}
f_{D_{l+3}}(t)  = (a_{l} + b_{l} t) f_{D_{l+2}}(t) 
 +(c_{l} + d_{l} t + e_{l} t^2) f_{D_{l+1}}(t) + (f_{l} + g_{l} t)  f_{D_l}(t). 
 \end{equation}
}
Here, $a_{l}$, $b_{l}$, $c_{l}$, $d_{l}$, $e_{l}$, $f_{l}$ and $g_{l}$ are the following rational functions:
 { \small
 \[
  a_{l} = \frac{(l+1)(5l^2+4l-21)}{(l-1)(l+3)(5l+4)},\,\,\,\,\,\,\,\,\,
  \]
  \vspace{0.05cm}
  \[
  b_{l} = -\frac{2(l+1)(5l^2+4l-21)}{(l-1)(l+3)(5l+4)},
  \]
   \vspace{0.05cm}
 \[
  c_{l} = \frac{l(5l^2+14l+5)}{(l-1)(l+3)(5l+4)},\,\,\,\,\,\,\,\,
  \]
  \vspace{0.05cm}
  \[
  d_{l} = -\frac{4l(2l+1)(5l+9)}{(l-1)(l+3)(5l+4)},\,\,\,
  \]
  \vspace{0.05cm}
  \[
  e_{l} =\frac{4l(2l+1)(5l+9)}{(l-1)(l+3)(5l+4)},\,\,\,\,\,\,
 \]
   \vspace{0.05cm}
 \[
  f_{l} = -\frac{(l+1)(5l+9)}{(l+3)(5l+4)},\,\,\,\,\,\,\,\,\,\,\,\,\,\,\,\,\,\,
  \]
  \vspace{0.05cm}
 \[
  g_{l} =\frac{2(l+1)(5l+9)}{(l+3)(5l+4)}.\,\,\,\,\,\,\,\,\,\,\,\,\,\,\,\,\,\,\,
  \]
  }

% { \small
% \[
%  a_{l} = \frac{(l + 2)(43l^3- 78l^2- 129l-24)}{(l + 3)(43l^3- 35l^2- 36l-32)},\,\,\,\,\,\,\,\,\,\,\,
%  \]
%  \[
%  b_{l} = -\frac{86l^4+ 145l^3- 196l^2- 623l-456}{(l + 3)(43l^3 - 35l^2- 36l - 32)},
%  \]
%  \[
%  c_{l} = \frac{l(43l^3+ 180l^2+ 45l+56)}{(l+3)(43l^3- 35l^2- 36l-32)},\,\,\,\,\,\,\,\,\,\,\,\,
%  \]
%  \[
%  d_{l} = -\frac{2l(172l^3+ 333l^2- 23l-32)}{(l+3)(43l^3- 35l^2- 36l-32)},\,\,\,\,\,\,\,\,
%  \]
%  \[
%  e_{l} =\frac{2(2l-1)(2l+1)(43l^2 + 51l - 24)}{(l+3)(43l^3- 35l^2- 36l-32)},\,\,\,\,
% \]
%  \[
%  f_{l} = -\frac{(l-1)(43l^3+ 137l^2+ 38l-48)}{(l+3)(43l^3- 35l^2- 36l-32)},\,\,\,\,\,
%  \]
% \[
%  g_{l} =\frac{(l-1)(2l+1)(43l^2+ 51l-24)}{(l+3)(43l^3- 35l^2- 36l-32)}.\,\,\,\,\,\,\,\,\,\,
%  \]
%  }
\end{theorem}
\begin{proof}

Let us consider the $k$th coefficient of $f_{D_{l}}(t)$ up to the sign $(-1)^k$:
\[
\begin{array}{rcl}
\mathcal{C}(l, k)  &:= &\binom {l}{k}\binom {l+k-1}{k} + \binom {l-2}{k-2}\binom {l+k-2}{k}\\
& =&  \frac{(l+k-2)!}{(l-k)!(k!)^2}\big\{  l(l+k-1) + k(k-1) \big\}.
\end{array}
\]
We compute the coefficient of the term $(-t)^k$ on the right hand side of \eqref{recurrenceD}.
 {\small
 \[
\begin{array}{ll}
 a_{l}\cdot \mathcal{C}(l+2, k)- b_{l}\cdot \mathcal{C}(l+2, k-1)+ c_{l}\cdot \mathcal{C}(l+1, k)- d_{l}\cdot \mathcal{C}(l+1, k-1) \\
\vspace{0.05cm}
 + e_{l}\cdot \mathcal{C}(l+1, k-2) + f_{l}\cdot \mathcal{C}(l, k) - g_{l}\cdot \mathcal{C}(l, k-1) \\
\vspace{0.05cm}
=  \large \frac{(l+k-3)!}{(l+3-k)!(k!)^2(l-1)(l+3)(5l+4)}  \Bigl[ (l + 1)(5 l^2 + 4 l - 21)(l + k)(l + k - 1)(l + k - 2)\\
\vspace{0.05cm}
 \times(l - k + 3)\{(l + 2)(l + k + 1) + k (k - 1)\}\\
\vspace{0.05cm}
 + 2 (l + 1)(5 l^2 + 4 l - 21)k^{2}(l + k - 1)(l + k - 2)\\
\vspace{0.05cm}
 \times\{(l + 2)(l + k) + (k - 1)(k - 2)\} \\
\vspace{0.05cm}
 + l(5 l^2 + 14 l + 5)(l - k + 3)(l - k + 2)(l + k - 1)(l + k - 2)\\
\vspace{0.05cm}
 \times \{(l + 1)(l + k) + k(k - 1)\} \\
\vspace{0.05cm}
 +  4 l(2 l + 1)(5 l + 9) k^{2}(l + k - 2)(l - k +  3)\\
\vspace{0.05cm}
 \times \{(l + 1)(l + k - 1) + (k - 1)(k - 2)\} \\
\vspace{0.05cm}
 + 4 l(2 l + 1)(5 l + 9) k^{2}(k - 1)^{2}\\
\vspace{0.05cm}
 \times \{(l + 1)(l + k - 2) + (k - 2)(k - 3)\} \\
\vspace{0.05cm}
 - (l -  1)(l + 1)(5 l + 9)(l + k - 2)(l - k + 3)(l - k + 2)(l - k +  1)\\
\vspace{0.05cm}
 \times \{l(l + k - 1) + k(k - 1)\}\\
\vspace{0.05cm}
 - 2 (l - 1)(l + 1)(5 l + 9)(l - k + 3)(l - k + 2)k^2\\
\vspace{0.05cm}
 \times \{l(l + k - 2) + (k - 1)(k - 2)\} \Bigr] \\
\vspace{0.05cm}
= \frac{(l+k+1)!}{(l+3-k)!(k!)^2} (6+2k+ k^{2}+5l +kl+l^{2}) \\
\vspace{0.05cm}
= \mathcal{C}(l+3, k).
\end{array}
\]}

\end{proof}
As an application of the recurrence relation, we observe the following.

\noindent
\begin{corollary} For each types $P_{l} = A_l (l\ge1), B_l (l\ge2)$, the $f$-polynomial $f_{P_{l}}(t)$ is divisible by $2t-1$ if and only if $l$ is odd.  
\end{corollary}
Therefore, due to (2.3), we obtain the following.
\begin{remark} The $f$-polynomial $f_{D_{l}}(t)$ ($l\ge4$) is divisible by $2t-1$ if and only if $l$ is odd. 
\end{remark}
Although for each types $P_{l} = A_l (l\ge1), B_l (l\ge2)$ the $f$-polynomial $f_{P_{l}}(t)$ is a solution of the Gauss hypergeometric differential equation, the $f$-polynomial $f_{D_{l}}(t)$ is a solution of the following Fuchsian equation of third-order.
\begin{remark} The $f$-polynomial $f_{D_{l}}(t)$ satisfies the following Fuchsian ordinary differential equation of third-order. The proof is left to the reader.
{\footnotesize\[
 t(t-1)(2t-1)\frac{ {\mathrm{d}}^{3}y }{ \mathrm{d}t^3 } + \{(l+6)(t^2-t)+2\}\frac{ {\mathrm{d}}^{2}y }{ \mathrm{d}t^2 }-l(l-1)(2t-1)\frac{ \mathrm{d}y }{ \mathrm{d}t }-l(l-1)(l+2)y=0.
\]}
\end{remark}
\section{Proof of Conjecture 2 except for types $D_l$ }

In the present section, we prove, except for types $D_l$, the following theorem, which approves Conjecture 2. The proof for types $D_l$ will be given in the next section 5. 

\medskip
\begin{theorem} 
\label{rootsP}
 The $f$-polynomial $f_{P}(t)$ for any finite type $P$ has $\rank(P)$ simple roots on the interval $(0, 1)$. 
\end{theorem}

\begin{proof}
{\bf Case I:  type} $A_l$ ($l\in \Z_{\ge1}$) and $B_l$ ($l\in\Z_{\ge1}$).   

This is an immediate consequence of the formulae \eqref{JacobiA} and \eqref{JacobiB}, since the Jacobi polynomials $\widetilde{P}_{l}^{(1, 1)}$ and $\widetilde{P}_{l}^{(0, 0)}$ are well known to have $l$ simple roots on the interval $(0, 1)$  (see \cite{[Sz]} Theorem 3.3.1).

\medskip
\noindent
{\bf Case II: Exceptional types and non-crystallographic types}    
% $P = E_6, E_7, E_8, F_4,$ $G_2, H_3, H_4$ and $I_2(p)$.

We apply the Euclid division algorithm for the pair of polynomials $f_0:= f_{P}$ and $f_1:=f'_{P}$. So, we obtain, a sequence $f_0,\ f_1,\ f_2, \ldots$ of polynomials in $t$ such that $f_{k-1}=f_k\cdot q_{k-1}+f_{k+1}$ for $k=1, 2, \ldots$ (where $q_{k-1}$ is the quotient and $f_{k+1}$ is the remainder). 

Then, we prove the following fact by direct calculations case by case.

\begin{fact} 
i) The degrees of the sequence $f_0,\ f_1,\ f_2, \ldots$ of polynomials descend one by one, and $f_{l}$ is a non-zero constant.

ii) The sequence $f_0(0), f_1(0), -f_2(0), \ldots, (-1)^{l-1}f_{l}(0)$ has constant sign and the sequence $f_0(1), f_1(1), -f_2(1), \ldots, (-1)^{l-1}f_{l}(1)$ has alternating sign.
\end{fact}

Applying the Sturm theorem (see for instance \cite{[T]} Theorem 3.1), we observe that $f_0$ has $l$ distinct roots on the interval $(0, 1)$.

This completes a proof of Theorem \ref{rootsP}.
\end{proof}

\section{Proof of Conjecture 2 for types $D_l \, (l\ge4)$}

In this section, we prove the following theorem, which answers to Conjecture 2 for the types $D_l$ ($l\ge4$) affirmatively.

\begin{theorem}
\label{rootsD}
The polynomial $f_{D_l}(t)$ has $l$ simple roots on the interval $(0, 1)$.
\end{theorem}
\begin{proof}
First, we recall that the shifted Legendre polynomial {\small$\widetilde{P}_{l}^{(0, 0)}(t)$} ({\small$= (-1)^{l}f_{B_l}(t)$}) satisfies the functional equation {\small$\widetilde{P}_{l}^{(0, 0)}(t)=(-1)^{l}\widetilde{P}_{l}^{(0, 0)}(1-t)$}. Hence, thanks to (2.3),.
\begin{lem}
$f_{D_l}(t)=(-1)^{l}f_{D_l}(1-t)$.
\end{lem}
Secondly, let $t_{B_{l}, \nu}, \nu=1, 2, \ldots, l$, be the zeros of $f_{B_l}(t)$ in decreasing order (i.e. $1 > t_{B_{l}, 1} > t_{B_{l}, 2}  > \cdots > t_{B_{l}, l} > 0$). Then, the following fact is known (see \cite{[Sz]} Theorem 3.3.2.).
\begin{fact} 
The system $\{ t_{B_{l}, \nu} \}_{\nu=1}^{l}$ alternates with the system $\{ t_{B_{l+1}, \nu} \}_{\nu=1}^{l+1}$, that is,
\[
t_{B_{l+1}, \nu} > t_{B_{l}, \nu} > t_{B_{l+1}, \nu+1}, \,\,\,\,(\nu= 1, \ldots, l).
\]
\end{fact}
\noindent
{\bf Case I: }$l = 2k$\\   
We consider $2k$ open intervals $(t_{B_{2k}, 2k+1-\nu}, t_{B_{2k-1}, 2k-\nu}) (\nu=1, 2, \ldots, k)$ and \\$(t_{B_{2k-1}, k+1-\nu}, t_{B_{2k}, k+1-\nu}) (\nu=1, 2, \ldots, k)$. We note that $t_{B_{2k-1}, k} = 1/2$. On the intervals $(t_{B_{2k}, 2k+1-\nu}, t_{B_{2k-1}, 2k-\nu}) (\nu=1, 2, \ldots, k)$, the polynomials {\small$f_{B_{2k}}(t)$} and {\small$f_{B_{2k-1}}(t)$} have the opposite sign. Moreover, due to the identity (2.3), we can show\\ {\small$f_{D_{2k}}(t_{B_{2k}, 2k+1-\nu})f_{D_{2k}}(t_{B_{2k-1}, 2k-\nu}) < 0$}, $\nu=1, \ldots, k$. Thanks to intermediate value theorem, for each interval $(t_{B_{2k}, 2k+1-\nu}, t_{B_{2k-1}, 2k-\nu}) (\nu=1, 2, \ldots, k)$, there exists at least one root of $f_{D_{2k}}(t)$. Due to Lemma 5.2, we have that for each interval $(t_{B_{2k-1}, k+1-\nu}, t_{B_{2k}, k+1-\nu})  (\nu=1, 2, \ldots, k)$, there exists at least one root of {\small$f_{D_{2k}}(t)$}. Since the polynomial $f_{D_{2k}}(t)$ is of precise degree $2k$, we conclude that, in each intervals $(t_{B_{2k}, 2k+1-\nu}, t_{B_{2k-1}, 2k-\nu}) (\nu=1, 2, \ldots, k)$ and $(t_{B_{2k-1}, k+1-\nu}, t_{B_{2k}, k+1-\nu})\\ (\nu=1, 2, \ldots, k)$, there is one and only one root of the polynomial {\small$f_{D_{2k}}(t)$}.\\
\medskip
\noindent
{\bf Case II:} $l = 2k+1$\\    
In a similar manner to Case I, we conclude that the polynomial {\small$f_{D_l}(t)$} has $l$ simple roots on the interval $(0, 1)$.
% $P = E_6, E_7, E_8, F_4,$ $G_2, H_3, H_4$ and $I_2(p)$.

\end{proof}
From the discussion in the proof of Theorem 5.1, we have shown the following.
\begin{corollary}
1.  For the case where $l = 2k$, we obtain the following properties for $\nu = 1, 2, \ldots, k$:
\[
t_{D_{2k}, 2k+1-\nu}\in (t_{B_{2k}, 2k+1-\nu}, t_{B_{2k-1}, 2k-\nu}),
\]
\[
t_{D_{2k}, k+1-\nu}\in (t_{B_{2k-1}, k+1-\nu}, t_{B_{2k}, k+1-\nu}).
\]
2.  For the case where $l = 2k+1$, we obtain the following properties for $\nu = 1, 2, \ldots, k$:
\[
t_{D_{2k+1}, 2k+2-\nu}\in (t_{B_{2k+1}, 2k+2-\nu}, t_{B_{2k}, 2k+1-\nu}),\, t_{D_{2k+1}, k+1}=1/2,
\]
\[
t_{D_{2k+1}, k+1-\nu}\in (t_{B_{2k}, k+1-\nu}, t_{B_{2k+1}, k+1-\nu}).\,\,\,\,\,\,\,\,\,\,\,\,\,\,\,\,\,\,\,\,\,\,\,\,\,\,\,\,\,\,\,\,\,\,\,\,\,\,\,\,\,\,\,\,\,\,\,\,\,\,\,\,\,\,
\]
\end{corollary}
\begin{remark}
For an irreducible (possibly noncrystallographic) root system of rank $l$ of type $P$, F. Chapoton defined Chapoton's $F$-triangle\footnote{In the definition, the sign of variables changes from the original definition by F. Chapoton.} (\cite{[C]}) of type $P$ as
\begin{equation}
F_{P}(x, y):= \sum_{k=0}^{l}\sum_{m=0}^{l}f_{k, m}(-x)^{k}(-y)^{m}
\end{equation}
, where $f_{k, m}$ is the number of faces of the cluster complex $\Delta(P)$ consisting of $k$ positive roots and $m$ negative simple roots. Clearly, $f_{k, m}=0$ unless $k+m \le l$. We note that $f^{+}_{P}(x) = F_{P}(x, 0)$ and $f_{P}(x) = F_{P}(x, x)$. In \cite{[Sa4]}, the author studied the zero locus of Chapoton's $F$-triangle. Based on some numerical experiments, the author gave some conjectures on the zeros of $F$-triangle.
\end{remark}
%For the Rodrigues type formula of type $D_l$ for $l \in \Z_{\ge4}$, we consider% the factor in the derivatives:
%\[
%\begin{array}{rcl}
%G_{l}(t) &:=&  (t^2-t)^{l-2} \Bigl\{(l-1) - (3l-2)t + (3l-2)t^2 \Bigr\}\\
%\end{array}
%\]
%so that the following equality holds.
%\begin{equation}
%\label{RodriguesD3}
%\begin{array}{rcl}
%f_{D_l}(t)=\frac{(-1)^{l}}{(l-1)!}\big(\frac{d}{dt}\big)^{l-2} \big( G_l(t)\big%)
%\end{array}
%\end{equation}
%Set $G^{(i)}_{l}(t) := \frac{\mathrm{d}^{i}}{\mathrm{d}t^{i}}G_{l}(t)$ for $0\l%e i\le l-2$.\\
%\noindent
%{\bf Formula A.}  According to the residue class $l\bmod 4$, we have 
%\[
%\begin{array}{rcl}
%\vspace{0.2cm}
%G_{4L}^{(2L)}(\frac{1}{2}) & =& (-1)^{L+1}2^{3-6L}\frac{(4L-1)!(2L)!}{(3L-1)!L!%},\\
%\vspace{0.2cm}
%G_{4L+1}^{(2L)}(\frac{1}{2}) & = & (-1)^{L}2^{-6L}\frac{(4L)!(2L)!}{(3L)!L!}, \%\
%\vspace{0.2cm}
%G_{4L+2}^{(2L+2)}(\frac{1}{2}) & = &(-1)^{L}2^{2-6L}\frac{(4L+1)!(2L+2)!}{(3L)!%(L+1)!}, \\
%G_{4L+3}^{(2L+2)}(\frac{1}{2}) & = &(-1)^{L+1}2^{-2-6L}3\frac{(4L+2)!(2L+2)!}{(%3L+1)!(L+1)!}.\qquad\qquad
%\end{array}
%\]

\section{Proof of Conjecture 3 }

In this section, we prove the following theorem, which approves Conjecture 3.\\
Let us fix notation: for $P\in${\small $\{A_l\ (l\!\ge\!1), B_l\ (l\! \ge\!1), D_l\ (l\! \ge\! 4), E_l\ (l=6, 7, 8), F_4, G_2, H_3,$ $H_4, I_2(p)\ (p\ge3)\}$}, let $t_{P, \nu}, \nu=1, 2, \ldots, l=\rank(P)$, be the zeros of {\small$f_{P}(t)$} in decreasing order (i.e. $1 > t_{P, 1} > t_{P, 2}  > \cdots > t_{P, l} > 0$).\\

\begin{theorem} 
\label{smallest}
For each series of types $P_{l} = A_l, B_l, D_l$, the smallest zero locus of {\small$f_{P_l}(t)$} monotonously decreasingly converges to zero as the rank $l$ tends to infinity.
\end{theorem}
\begin{proof}
\smallskip
\noindent
Type $A_l$: Due to Theorem 6.1 in \cite{[I-S]}, the smallest zero locus of {\small$f^{+}_{A_{l+1}}(t)$} monotonously decreasingly converges to zero as the rank $l$ tends to infinity. Hence, from (2.1), we show that the smallest zero locus of $f_{A_l}(t)$ monotonously decreasingly converges to zero.\\
\smallskip
\noindent
Type $B_l$: Recall a fact on the distribution of the zeros of {\small$\widetilde{P}_{l}^{(0, 0)}(t)$} ({\small$= (-1)^{l}f_{B_l}(t)$}) (\cite{[Sz]} Theorem 6.21.3).
\begin{fact}
\label{Bruns}
%1. The shifted Legendre polynomial $\widetilde{L}_{l}(t)$ of degree $l$ has $l$ number of simple real roots on the interval $(0, 1)$.\\
 Let $\tilde{x}_{\nu} = \tilde{x}_{l, \nu}, \nu=1, 2, \ldots, l$, be the zeros of $\widetilde{P}_{l}^{(0, 0)}(t)$ in decreasing order.    
 %(i.e. $1 > \tilde{x}_{1} > \tilde{x}_{2}  > \cdots > \tilde{x}_l > 0$).
  Let $\theta_l= \theta_{l, \nu} \in (0, \pi), \nu=1, 2, \ldots, l$, be the real number defined by
\[
\cos \theta_{\nu}= 2 \tilde{x}_{\nu} - 1.
\]
Then, the inequalities hold as follows:\quad 
$ %\[
\frac{\nu - \frac{1}{2}}{l + \frac{1}{2}}\pi < \theta_{\nu} < \frac{\nu}{l + \frac{1}{2}}\pi \,\,\,\,\quad (\nu=1, 2, \ldots, l). 
$ %\]
\end{fact}
Therefore, we have that the smallest zero locus of {\small$f_{B_l}(t)$} monotonously decreasingly converges to zero.\\
\smallskip
\noindent
Type $D_l$: Due to Corollary 5.4, we show that the smallest zero locus $t_{D_{l}, l}$ of {\small$f_{D_l}(t)$} is an element of the interval $(t_{B_{l}, l}, t_{B_{l-1}, l-1})$. Therefore, we conclude that the smallest zero locus $t_{D_{l}, l}$ monotonously decreasingly converges to zero.
\end{proof}

\section{Proof of Conjecture 4 except for types $D_l$}

In this section, we prove, except for types $D_l$, the following theorem, which approves Conjecture 4. The proof for types $D_l$ will be given in the next section 8.
Let us fix notation: for $P\in${\small $\{A_l\ (l\!\ge\!1), B_l\ (l\! \ge\!1), D_l\ (l\! \ge\! 4), E_l\ (l=6, 7, 8), F_4, G_2, H_3,$ $H_4, I_2(p)\ (p\ge3)\}$}, let $t^{+}_{P, \nu}, \nu=1, 2, \ldots, l=\rank(P)$, be the zeros of $f^{+}_{P}(t)$ in decreasing order (i.e. $1 = t^{+}_{P, 1} > t^{+}_{P, 2}  > \cdots > t^{+}_{P, l} > 0$).\\
\begin{theorem} 
The following inequalities hold for any type $P$
\[
t_{P, \nu} > t^{+}_{P, \nu+1} >t_{P, \nu+1} , \,\,\,\,(\nu= 1, \ldots, l-1).
\]

\end{theorem}
\begin{proof}
\noindent
{\bf I.  Case for types $A_l$ and $B_l$.}  \\
First, for type $A_l$, we recall the following fact (see \cite{[Sz]} Theorem 3.3.2.).
\begin{fact} 
The system $\{ t^{+}_{A_{l}, \nu} \}_{\nu=2}^{l}$ alternates with the system $\{ t^{+}_{A_{l+1}, \nu} \}_{\nu=2}^{l+1}$, that is,
\[
t^{+}_{A_{l+1}, \nu} > t^{+}_{A_{l}, \nu} > t^{+}_{A_{l+1}, \nu+1}, \,\,\,\,(\nu= 2, \ldots, l).
\]
\end{fact}
Hence, from (2.1), we have the following inequalities
\[
t_{A_{l}, \nu} > t^{+}_{A_{l}, \nu+1} > t_{A_{l}, \nu+1}, \,\,\,\,(\nu= 1, \ldots, l-1).
\]
Next, we recall {\small$f_{B_l}(t) = (-1)^{l}\widetilde{P}_{l}^{(0, 0)}(t)$}. Due to Proposition 6.6 in \cite{[I-S]}, we have the following inequalities
\[
t_{B_{l}, \nu} > t^{+}_{B_{l}, \nu+1} > t_{B_{l}, \nu+1}, \,\,\,\,(\nu= 1, \ldots, l-1).
\]
\noindent
{\bf II.  Exceptional types and non-crystallographic types.}  \\
In section 4, for $P\in${\small $\{ E_l\ (l=6, 7, 8), F_4, G_2, H_3, H_4, I_2(p)\ (p\ge3)\}$}, we constructed a Strum sequence $f_0(t), f_1(t), -f_2(t), \ldots, (-1)^{l-1}f_{l}(t)$ on $[0, 1]$. In \cite{[I-S]}\S4, for the pair of polynomials $f^{+}_0:=f^{+}_{P}(t)/(1-t)$ and $f^{+}_1:=(f_{0})'$, we constructed a Strum sequence\\ $f^{+}_0(t), f^{+}_1(t), -f^{+}_2(t), \ldots, (-1)^{l-2}f^{+}_{l-1}(t)$ on $[0, 1]$. For $t_{0} \in [0, 1]$, \\let $V(f_{P}, t_{0})$ (resp.~$V(f^{+}_{P}, t_{0})$) be the number of sign changes in the sequence\\ $f_0(t), f_1(t), -f_2(t), \ldots, (-1)^{l-1}f_{l}(t)$ (resp.~$f^{+}_0(t), f^{+}_1(t), -f^{+}_2(t), \ldots, (-1)^{l-2}f^{+}_{l-1}(t)$). \\For $t_{0} \in [0, 1]$, we put
\[
\overline{V}(f_{P}, t_{0}):= V(f_{P}, 0) - V(f_{P}, t_{0}),\,\, \overline{V}(f^{+}_{P}, t_{0}):= V(f^{+}_{P}, 0) - V(f^{+}_{P}, t_{0}).
\]
Then, we prove the following fact case by case.

\begin{fact} 
There exist sequences $\{ \alpha_{P, \nu} \}_{\nu=1}^{l-1}$ and $\{ \alpha^{+}_{P, \nu} \}_{\nu=1}^{l-1}$ of real numbers which satisfy inequalities $0 < \alpha_{P, l-1} < \alpha^{+}_{P, l-1}<\cdots <\alpha_{P, 1} < \alpha^{+}_{P, 1} < 1$ such that
\[
\overline{V}(f_{P}, \alpha_{P, i})= l-i , \overline{V}(f_{P}, \alpha^{+}_{P, i})= l-i,\,\,\,\,\,\,\,\,\,\,\,\,\,\,\,\,\,\,\,\,\,\,\,\,\,\,\,\,\,\,\,\,\,\,\,\,\,\,\,\,\,\,\,\,\,\,\,\,\,\,\,\,\,\,\,\,\,
\]
\[
\overline{V}(f^{+}_{P}, \alpha_{P, i})= l-i-1 , \overline{V}(f^{+}_{P}, \alpha^{+}_{P, i})=l-i,(i=1, \ldots, l-1).
\]
\end{fact}
Therefore, due to the Sturm theorem (see for instance \cite{[T]} Theorem 3.1), we have that $t_{P, \nu} \in (\alpha^{+}_{P, \nu}, \alpha_{P, \nu-1})$ and $t^{+}_{P, \nu} \in (\alpha_{P, \nu-1}, \alpha^{+}_{P, \nu-1})$ ($\nu=2, \ldots, l$), where we put $\alpha^{+}_{P, l}:=0$. Hence, we have the following inequalities
\[
t_{P, \nu} > t^{+}_{P, \nu+1} >t_{P, \nu+1} , \,\,\,\,(\nu= 1, \ldots, l-1).
\]
\end{proof}
\smallskip
\noindent
{\bf Example}\,\,For each $P\in${\small $\{ E_l\ (l=6, 7, 8), F_4, G_2, H_3, H_4, I_2(p)\ (p\ge3)\}$}, we give an example of two kinds of sequences $\{ \alpha_{P, \nu} \}_{\nu=1}^{l-1}$ and $\{ \alpha^{+}_{P, \nu} \}_{\nu=1}^{l-1}$ that satisfy the inequalities in Fact 7.3.\\
$E_{6}$\,:\\
$\alpha_{E_{6}, 5}=7/200, \alpha^{+}_{E_{6}, 5}=1/10, \alpha_{E_{6}, 4}=21/100, \alpha^{+}_{E_{6}, 4}=1/4, \alpha_{E_{6}, 3}=2/5,$\\
$ \alpha^{+}_{E_{6}, 3}=3/5, \alpha_{E_{6}, 2}=13/20, \alpha^{+}_{E_{6}, 2}=7/10, \alpha_{E_{6}, 1}=17/20, \alpha^{+}_{E_{6}, 1}=9/10$.\\
$E_{7}$\,:\\
$\alpha_{E_{7}, 6}=1/50, \alpha^{+}_{E_{7}, 6}=1/10, \alpha_{E_{7}, 5}=3/20, \alpha^{+}_{E_{7}, 5}=1/5, \alpha_{E_{7}, 4}=8/25,$\\
$\alpha_{E_{7}, 4}=2/5, \alpha^{+}_{E_{7}, 3}=13/25, \alpha_{E_{7}, 3}=3/5, \alpha^{+}_{E_{7}, 2}=7/10, \alpha_{E_{7}, 2}=4/5,$\\
$\alpha^{+}_{E_{7}, 1}=87/100, \alpha_{E_{7}, 1}=19/20.$\\
$E_{8}$\,:\\$\alpha_{E_{8}, 7}=19/2000, \alpha^{+}_{E_{8}, 7}=1/100, \alpha_{E_{8}, 6}=11/100, \alpha^{+}_{E_{8}, 6}=1/5, \alpha_{E_{8}, 5}=1/4,$\\
$\alpha_{E_{8}, 5}=3/10, \alpha^{+}_{E_{8}, 4}=21/50, \alpha_{E_{8}, 4}=49/100, \alpha^{+}_{E_{8}, 3}=3/5, \alpha_{E_{8}, 3}=7/10,$\\
$\alpha_{E_{8}, 2}=77/100, \alpha^{+}_{E_{8}, 2}=4/5, \alpha_{E_{8}, 1}=9/10, \alpha^{+}_{E_{8}, 1}=19/20.$\\
$F_{4}$\,:\\
$\alpha_{F_{4}, 3}=1/20, \alpha^{+}_{F_{4}, 3}=1/10, \alpha_{F_{4}, 2}=7/20, \alpha^{+}_{F_{4}, 2}=2/5, \alpha_{F_{4}, 1}=7/10, \alpha^{+}_{F_{4}, 1}=4/5.$\\
$G_{2}$\,:\,$\alpha_{G_{2}, 1}=1/6, \alpha^{+}_{G_{2}, 1}=1/2.$\\
$H_{3}$\,:\,$\alpha_{H_{3}, 2}=7/100, \alpha^{+}_{H_{3}, 2}=1/10,\alpha_{H_{3}, 1}=11/20, \alpha^{+}_{H_{3}, 1}=3/5.$\\
$H_{4}$\,:\,$\alpha_{H_{4}, 3}=9/500, \alpha^{+}_{H_{4}, 3}=1/5,\alpha_{H_{4}, 2}=31/100, \alpha^{+}_{H_{4}, 2}=2/5, \alpha_{H_{4}, 1}=7/10,$\\
$ \alpha^{+}_{H_{4}, 1}=4/5.$\\
$I_{2}(p\ge3)$:\,$\alpha_{I_{2}(p), 1}=1/p, \alpha^{+}_{I_{2}(p), 1}=1/2.$\\

\section{Proof of Conjecture 4 for types $D_l \ \ (l\ge4)$}
In the case where $l=4$, we approve Conjecture 4 by hand calculation. Throughout this section, we assume $l\ge5$. 
\begin{theorem} The following inequalities hold for $\nu= 1, \ldots, l-1$
\[
t_{D_{l}, \nu} > t^{+}_{D_{l}, \nu+1} >t_{D_{l}, \nu+1} .
\]
\end{theorem}
\begin{proof}
First, we recall an equation from \cite{[I-S]}\S5
{\footnotesize\[
f^{+}_{D_l}(t)=\frac{(-1)^{l-3}}{(l-2)!}\big((1-t)H^{(l-3)}_{l}(t)-(l-3)H^{(l-4)}_{l}(t) \big),
\]}
where {\small$H_{l}(t) :=  (t^2-t)^{l-3} \Bigl\{(l-2) - (3l-4)t + (3l-4)t^2 \Bigr\}$} and its higher order derivatives {\small$H^{(i)}_{l}(t) := \frac{\mathrm{d}^{i}}{\mathrm{d}t^{i}}H_{l}(t)$} for $0\le i\le l-3$. From Lemma 5.2 in \cite{[I-S]}, let  $1=u_1>u_2>\cdots> u_{l-1}> u_l=0$ be all roots of the polynomial $H_l^{(l-4)}(t)=0$. Let $1>v_1>v_2>\cdots>v_{l-1}>0$ be the $l-1$ roots of $H_l^{(l-3)}(t)=0$ so that one has the inequalities:
\[
u_1 > v_1 > u_2 > v_2 > \cdots > u_{l-1} > v_{l-1} > u_l.
\]
Moreover, in the proof of Lemma 5.2 in \cite{[I-S]}, we have shown the following fact.
\begin{fact} 
$t^{+}_{D_{l},l+1-\nu} \in (u_{l+1-\nu},v_{l-\nu}), \,\,\,\,(\nu= 1, \ldots, l-1)$ and $t^{+}_{D_{l}, 1}=1$.
\end{fact} 
Next, the proof for Theorem 8.1 is divided into two parts. \\
\noindent
{\bf Part I}
We discuss location of the following roots
\[
t_{D_{l}, l+1-\nu}\,\, \mathrm{and}\,\, t^{+}_{D_{l}, l+1-\nu}\,(\nu= 1, \ldots, \lfloor l/2 \rfloor).
\]
First, we define two kinds of functions
{\small\[
\tilde{f}_{D_{l}}(t) :=\frac{(-1)^{l-3}}{(l-2)!}\big((1-2t)H^{(l-3)}_{l}(t)-2(l-3)H^{(l-4)}_{l}(t) \big),\,\,\,\,\,\,\,\,\,\,\,\,\,\,\,\,\,\,\,\,\,\,\,\,
\]}
{\small\[
K_{l}(t) := f^{+}_{D_{l}}(t) - \tilde{f}_{D_{l}}(t)=\frac{(-1)^{l-3}}{(l-2)!}\big(tH^{(l-3)}_{l}(t)+(l-3)H^{(l-4)}_{l}(t) \big).
\]}
For the polynomial {\small$\tilde{f}_{D_{l}}(t)$}, we have the following proposition.
\begin{proposition} 
1. The following identity holds for $l = 5, 6, \ldots$:
{\footnotesize\begin{equation}
 \tilde{f}_{D_{l}}(t) = \frac{1}{(l-2)!}\frac{\mathrm{d}^{l-3}}{\mathrm{d}t^{l-3}}\biggl[ t^{l-3}(1-t)^{l-3}(1-2t) \Bigl\{(l-2) - (3l-4)t + (3l-4)t^2 \Bigr\} \biggr].
\end{equation}}
\\
2. The following identity holds for $l = 5, 6, \ldots$:
{\small\begin{equation}
\tilde{f}_{D_{l}}(t) = \frac{l+2}{l}f_{D_{l}}(t) - \frac{2}{l}f_{B_{l}}(t).
\end{equation}}
\\
3. The following identity holds for $l = 5, 6, \ldots$:
{\small\begin{equation}
\tilde{f}_{D_{l}}(t) =(-1)^{l}\tilde{f}_{D_{l}}(1-t).
\end{equation}}
\end{proposition}
\begin{proof}
1. By definition, we have
{\small\[
2f^{+}_{D_{l}}(t) =\tilde{f}_{D_{l}}(t) + \frac{(-1)^{l-3}}{(l-2)!}H^{(l-3)}_{l}(t).
\]}
We recall the Rodrigues type formula for {\small$f^{+}_{D_{l}}(t)$} from Theorem 5.1 in \cite{[I-S]}
{\small\[
f^{+}_{D_{l}}(t) = \frac{1}{(l-2)!}\frac{\mathrm{d}^{l-3}}{\mathrm{d}t^{l-3}}\biggl[ t^{l-3}(1-t)^{l-2} \Bigl\{(l-2) - (3l-4)t + (3l-4)t^2 \Bigr\} \biggr].
\]}
By calculating {\small$2f^{+}_{D_{l}}(t)-\frac{(-1)^{l-3}}{(l-2)!}H^{(l-3)}_{l}(t)$}, we obtain the result.\\
2. From (2.7) and (8.1), we have
{\footnotesize\[
\tilde{f}_{D_{l}}(t) - f_{D_{l}}(t) \,\,\,\,\,\,\,\,\,\,\,\,\,\,\,\,\,\,\,\,\,\,\,\,\,\,\,\,\,\,\,\,\,\,\,\,\,\,\,\,\,
\]}
{\footnotesize\[
= \frac{2}{(l-1)!}\frac{\mathrm{d}^{l-2}}{\mathrm{d}t^{l-2}}\biggl[ t^{l-1}(1-t)^{l-1} \biggr]
\]}
{\footnotesize\[
=\frac{2}{l}f_{D_{l}}(t) - \frac{2}{l}f_{B_{l}}(t).\,\,\,\,\,\,\,\,\,\,\,\,\,\,\,\,\,\,\,\,\,\,\,\,\,\,\,\,
\]}
This completes the proof.\\
3. Since {\small$f_{B_{l}}(t) =(-1)^{l}f_{B_{l}}(1-t)$} and {\small$f_{D_{l}}(t) =(-1)^{l}f_{D_{l}}(1-t)$}, we obtain the result.
\end{proof} 
\begin{remark} As a consequence of (8.2), the polynomial $\tilde{f}_{D_{l}}(t)$  is divisible by \\$2t-1$ if and only if $l$ is odd. 
\end{remark}
We recall a formula from \cite{[I-S]}\S5.\\
\noindent
{\bf Formula B.}\footnote
{ Although, in \cite{[I-S]}\S5, the index $i$ for $h^{(2i-1)}_{k}(t)$ (resp. $h^{(2i)}_{k}(t)$) ranges from $1$ to $\lfloor(k+1)/2\rfloor$\\ (resp. $\lfloor k/2\rfloor$), the range of the index $i$ can be extended to $0\le i \le k$.} Set $h_k^{(i)}(t):=\big(\frac{d}{dt}\big)^i(t(t-1))^k$ for $0\le i \le k$.  Then, we have
\[
\begin{array}{rcl}
h^{(2i-1)}_{k}(\frac{1}{2})&=& 0 \,\,(i = 1, \ldots, k),\\
h^{(2i)}_{k}(\frac{1}{2})& =& \big(\!-\!\frac{1}{4}\big)^{k-i} \frac{k!(2i)!}{(k-i)!i!}  \,(i = 1, \ldots, k).
\end{array}
\]
\noindent
{\it Proof of Formula B.}  By induction on $i$, we obtain the following equations
{\footnotesize\begin{equation}
 h^{(2i-1)}_{k}(t) = \sum_{j=1}^{i}\frac{k!(2i-1)!}{(i-j)!(2j-1)!(k-i-j+1)!}(t^2-t)^{k-i-j+1}(2t-1)^{2j-1}
\end{equation}}
{\footnotesize\begin{equation}
 h^{(2i)}_{k}(t)=\sum_{j=0}^{i}\frac{k!(2i)!}{(i-j)!(2j)!(k-i-j)!}(t^2-t)^{k-i-j}(2t-1)^{2j}.
\end{equation}}
These equations hold for $0\le i \le k$.\\
From (8.4) and (8.5), we obtain the following formula.\\
\smallskip
\noindent
{\bf Formula C.}  We have
\[
h^{(l-3)}_{l-3}(0) = (-1)^{l-1}(l-3)!,\,\,h^{(l-3)}_{l-2}(0) = 0.
\]

\begin{proposition}
The polynomial {\small$\tilde{f}_{D_{l}}(t)$} has $l$ simple roots on the interval $(0, 1)$. 
\end{proposition}
\begin{proof}
For the case where $l=2k$, we consider $2k$ open intervals $(t_{D_{2k}, 2k+1-\nu}, t_{D_{2k}, 2k-\nu})\\ (\nu=1, 2, \ldots, k-1)$, $(t_{D_{2k}, k+1}, 1/2)$ and $(1/2, t_{D_{2k}, k})$, $(t_{D_{2k}, k+1-\nu}, t_{D_{2k}, k-1-\nu}) \\(\nu=1, 2, \ldots, k-1)$. For a non-zero real number $\theta$, we put 
\[
\mathrm{sgn}(\theta):=\theta/|\theta|. 
\]
From the discussion in the proof of Theorem 5.1, we have shown the following properties
{\small\[
\mathrm{sgn}(f_{B_{2k}}(t_{D_{2k}, 2k+1-\nu}))=(-1)^{\nu},\,(\nu=1, 2, \ldots, k).
\]}
Moreover, due to {\bf Formula B}, we have 
{\footnotesize\[
f_{B_{2k}}(1/2)=\big(-\frac{1}{4}\big)^{k} \frac{(2k)!}{k!k!}  \,\, \mathrm{and}\,\, f_{D_{2k}}(1/2)=\big(-\frac{1}{4}\big)^{k} \frac{(2k-2)(2k-2)!}{k!(k-1)!}. 
\]}
Hence, we have 
{\small\[
\mathrm{sgn}(\tilde{f}_{D_{2k}}(1/2))=(-1)^{k}.
\]}
Due to the identity (8.2), we can show {\small$\tilde{f}_{D_{2k}}(t_{D_{2k}, 2k+1-\nu})\tilde{f}_{D_{2k}}(t_{D_{2k}, 2k-\nu}) < 0$}, ($\nu=1, \ldots, k-1$) and {\small$\tilde{f}_{D_{2k}}(t_{D_{2k}, k+1})\tilde{f}_{D_{2k}}(1/2) < 0$}. Thanks to intermediate value theorem, for each interval $(t_{D_{2k}, 2k+1-\nu}, t_{D_{2k}, 2k-\nu}) (\nu=1, 2, \ldots, k-1)$ and $(t_{D_{2k}, 2k+1}, 1/2)$, there exists at least one root of {\small$\tilde{f}_{D_{2k}}(t)$}. From (8.3), the set of roots of {\small$\tilde{f}_{D_{2k}}(t)$} are symmetric with respect to the reflection centered at $t=1/2$. Therefore, we have that for each interval $(t_{D_{2k}, k+1-\nu}, t_{D_{2k}, k-\nu})  \\(\nu=1, 2, \ldots, k-1)$ and $(1/2, t_{D_{2k}, k})$, there exists at least one root of {\small$\tilde{f}_{D_{2k}}(t)$}. Since the polynomial $\tilde{f}_{D_{2k}}(t)$ is of precise degree $2k$, we conclude that, in each intervals\\ $(t_{B_{2k}, 2k+1-\nu}, t_{B_{2k-1}, 2k-\nu})$ ($\nu=1, 2, \ldots, k-1$), $(t_{D_{2k}, 2k+1}, 1/2)$, $(1/2, t_{D_{2k}, k})$ and\\ $(t_{B_{2k-1}, k+1-\nu}, t_{B_{2k}, k+1-\nu}) (\nu=1, 2, \ldots, k-1)$, there is one and only one root of the polynomial $\tilde{f}_{D_{2k}}(t)$.\\
For the case where $l=2k+1$, in a similar manner to case where $l=2k$, we conclude that the polynomial $\tilde{f}_{D_l}(t)$ has $l$ simple roots on the interval $(0, 1)$.
\end{proof}
Let $\tilde{t}_{D_{l}, \nu}, \nu=1, 2, \ldots, l$, be the zeros of $\tilde{f}_{D_l}(t)$ in decreasing order (i.e. $1 > \tilde{t}_{D_{l}, 1} > \tilde{t}_{D_{l}, 2}  > \cdots > \tilde{t}_{D_{l}, l} > 0$).\\
\begin{proposition}
1. For the case where $l = 2k$, we obtain the following properties for $\nu = 1, 2, \ldots, k-1$:
\[
\tilde{t}_{D_{2k}, 2k+1-\nu}\in (t_{D_{2k}, 2k+1-\nu}, t_{D_{2k}, 2k-\nu}),\,\tilde{t}_{D_{2k}, k+1}\in (t_{D_{2k}, k+1}, 1/2),
\]
\[
\tilde{t}_{D_{2k}, k} \in (1/2, t_{D_{2k}, k}),\,\tilde{t}_{D_{2k}, k-\nu}\in (t_{D_{2k}, k+1-\nu}, t_{D_{2k}, k-\nu}).\,\,\,\,\,\,\,\,\,\,\,\,\,\,\,\,\,\,\,\,\,\,\,\,\,
\]
2.  For the case where $l = 2k+1$, we obtain the following properties for $\nu = 1, 2, \ldots, k-1$:
\[
\tilde{t}_{D_{2k+1}, 2k+2-\nu}\in (t_{D_{2k+1}, 2k+2-\nu}, t_{D_{2k+1}, 2k+1-\nu}),\, \tilde{t}_{D_{2k+1}, k+2}\in (t_{D_{2k+1}, k+2}, 1/2),
\]
\[
\tilde{t}_{D_{2k+1}, k+1}=1/2,\,\,\,\,\,\,\,\,\,\,\,\,\,\,\,\,\,\,\,\,\,\,\,\,\,\,\,\,\,\,\,\,\,\,\,\,\,\,\,\,\,\,\,\,\,\,\,\,\,\,\,\,\,\,\,\,\,\,\,\,\,\,\,\,\,\,\,\,\,\,\,\,\,\,\,\,\,\,\,\,\,\,\,\,\,\,\,\,\,\,\,\,\,\,\,\,\,\,\,\,\,\,\,\,\,\,\,\,\,\,\,\,\,\,\,\,\,\,\,\,\,\,\,\,\,\,\,\,\,\,\,\,\,\,\,\,\,\,\,\,\,\,\,\,\,\,\,\,\,\,\,\,\,\,\,\,\,
\]
\[
\tilde{t}_{D_{2k+1}, k}\in (1/2, t_{D_{2k+1}, k}),\,\tilde{t}_{D_{2k+1}, k-\nu}\in (t_{D_{2k+1}, k+1-\nu}, t_{D_{2k+1}, k-\nu}).\,\,\,\,\,\,\,\,\,\,\,\,\,\,\,\,\,\,\,\,\,\,\,\,\,\,\,\,\,\,\,\,
\]

\end{proposition}
\begin{proposition}
1. For the case where $l = 2k$, we obtain the following properties for $\nu = 1, 2, \ldots, k$:
\[
\tilde{t}_{D_{2k}, 2k+1-\nu}\in (u_{2k+1-\nu}, v_{2k-\nu}).
\]

2.  For the case where $l = 2k+1$, we obtain the following properties for $\nu = 1, 2, \ldots, k$:
\[
\tilde{t}_{D_{2k+1}, 2k+2-\nu}\in (u_{2k+2-\nu}, v_{2k+1-\nu}), \tilde{t}_{D_{2k+1}, k+1}=1/2.
\]

\end{proposition}
\begin{proof}
By applying intermediate value theorem to the polynomial 
{\footnotesize\[
\tilde{f}_{D_{l}}(t)=\frac{(-1)^{l-3}}{(l-2)!}\big((1-2t)H^{(l-3)}_{l}(t)-2(l-3)H^{(l-4)}_{l}(t) \big),
\]}
 we obtain the results. We omit details.
\end{proof}
\begin{proposition}
1. The polynomial $K_{l}(t)$ has $l$ simple roots on the interval $[0, 1)$.\\ 
2. Let $t_{K_{l}, \nu}, \nu=1, 2, \ldots, l$, be the zeros of $K_{l}(t)$ in decreasing order (i.e. $1 > t_{K_{l}, 1} > t_{K_{l}, 2}  > \cdots > t_{K_{l}, l} = 0$). For $\nu=1, 2, \ldots, l-1$, we have
\[
t_{K_{l}, l} = 0,\, t_{K_{l}, l-\nu} \in (v_{l-\nu}, u_{l-\nu}).
\]
3. {\small$\left.\frac{\mathrm{d}}{\mathrm{d}t}K_{l}(t)\right|_{t=0}=l-2$}.
\end{proposition}
\begin{proof}
1. - 2. By applying intermediate value theorem to the polynomial $K_{l}(t)$, we obtain the results. We omit details.\\
3. By definition, we have {\small$K'_{l}(0)=\frac{(-1)^{l-3}}{(l-3)!}H^{(l-3)}_{l}(0)$}.\\
Due to {\bf Formula C}, we have 
\[
H^{(l-3)}_{l}(0)=(l-2)h^{(l-3)}_{l-3}(0)+(3l-4)h^{(l-3)}_{l-2}(0)=(-1)^{l+1}(l-2)!.
\]
Hence, we obtain {\small$\left.\frac{\mathrm{d}}{\mathrm{d}t}K_{l}(t)\right|_{t=0}=l-2$}.
\end{proof}
\begin{proposition}
1. For the case where $l = 2k$, we obtain the following properties for $\nu = 1, 2, \ldots, k-1$:
\[
t^{+}_{D_{2k}, 2k+1-\nu}\in (\tilde{t}_{D_{2k}, 2k+1-\nu}, \tilde{t}_{D_{2k}, 2k-\nu}),\,t^{+}_{D_{2k}, k+1}\in(\tilde{t}_{D_{2k}, k+1}, 1/2).
\]

2.  For the case where $l = 2k+1$, we obtain the following properties for $\nu = 1, 2, \ldots, k-1$:
\[
t^{+}_{D_{2k+1}, 2k+2-\nu}\in (\tilde{t}_{D_{2k+1}, 2k+2-\nu}, \tilde{t}_{D_{2k+1}, 2k+1-\nu}),\, t^{+}_{D_{2k+1}, k+2} \in (\tilde{t}_{D_{2k+1}, k+2}, 1/2).
\]
\end{proposition}
\begin{proof}
For the case where $l=2k$, we consider $k$ open intervals $(\tilde{t}_{D_{2k}, 2k+1-\nu}, \tilde{t}_{D_{2k}, 2k-\nu})\\ (\nu=1, 2, \ldots, k-1)$, $(\tilde{t}_{D_{2k}, k+1}, 1/2)$. From Proposition 8.8, we have shown the following properties
{\small\[
\mathrm{sgn}(K_{2k}(\tilde{t}_{D_{2k}, 2k+1-\nu}))=(-1)^{\nu+1},\,(\nu=1, \ldots, k).
\]}
Since {\small$f^{+}_{D_{2k}}(t) = \tilde{f}_{D_{2k}}(t) + K_{2k}(t)$}, we have 
{\small\[
\mathrm{sgn}(f^{+}_{D_{2k}}(\tilde{t}_{D_{2k}, 2k+1-\nu}))=(-1)^{\nu+1},\,(\nu=1, \ldots, k).
\]}
Moreover, due to {\bf Formula B}, we have 
{\footnotesize\[
f^{+}_{D_{2k}}(1/2)= \big(-\frac{1}{4}\big)^{k} \frac{(2k-4)(2k-3)!}{k!(k-2)!}.
\]}
From \cite{[I-S]}\S5, we have 
{\small\[
\mathrm{sgn}(f^{+}_{D_{2k}}(v_{2k-\nu}))=(-1)^{\nu},\,(\nu=1, 2, \ldots, k).
\]}
From Fact 8.2, we obtain the following results for $\nu = 1, 2, \ldots, k-1$:
\[
t^{+}_{D_{2k}, 2k+1-\nu}\in (\tilde{t}_{D_{2k}, 2k+1-\nu}, v_{2k-\nu}),\,t^{+}_{D_{2k}, k+1}\in(\tilde{t}_{D_{2k}, k+1}, 1/2).
\]
From Proposition 8.7, we also have for $\nu = 1, 2, \ldots, k-1$:
\[
t^{+}_{D_{2k}, 2k+1-\nu}\in (\tilde{t}_{D_{2k}, 2k+1-\nu}, \tilde{t}_{D_{2k}, 2k-\nu}),\,t^{+}_{D_{2k}, k+1}\in(\tilde{t}_{D_{2k}, k+1}, 1/2).
\]
For the case where $l=2k+1$, we consider $k$ open intervals \\$(\tilde{t}_{D_{2k+1}, 2k+2-\nu}, \tilde{t}_{D_{2k+1}, 2k+1-\nu}) (\nu=1, 2, \ldots, k-1)$, $(\tilde{t}_{D_{2k+1}, k+2}, 1/2)$. From Proposition 8.8, we have shown the following properties
{\small\[
\mathrm{sgn}(K_{2k+1}(\tilde{t}_{D_{2k+1}, 2k+2-\nu}))=(-1)^{\nu+1},\,(\nu=1, \ldots, k).
\]}
Since {\small $f^{+}_{D_{2k+1}}(t) = \tilde{f}_{D_{2k+1}}(t) + K_{2k+1}(t)$}, we have 
{\small\[
\mathrm{sgn}(f^{+}_{D_{2k+1}}(\tilde{t}_{D_{2k+1}, 2k+2-\nu}))=(-1)^{\nu+1},\,(\nu=1, \ldots, k).
\]}
Moreover, due to {\bf Formula B}, we have 
{\footnotesize\[
f^{+}_{D_{2k+1}}(1/2)= \big(-\frac{1}{4}\big)^{k} \frac{(2k-1)!}{k!(k-1)!}.
\]}
From \cite{[I-S]}\S5, we have 
{\small\[
\mathrm{sgn}(f^{+}_{D_{2k+1}}(v_{2k+1-\nu}))=(-1)^{\nu},\,(\nu=1, 2, \ldots, k).
\]}
From Fact 8.2, we obtain the following results for $\nu = 1, 2, \ldots, k-1$:
\[
t^{+}_{D_{2k+1}, 2k+2-\nu}\in (\tilde{t}_{D_{2k+1}, 2k+2-\nu}, v_{2k+1-\nu}),\,t^{+}_{D_{2k+1}, k+2}\in(\tilde{t}_{D_{2k+1}, k+2}, 1/2).
\]
From Proposition 8.7, we also have for $\nu = 1, 2, \ldots, k-1$:
\[
t^{+}_{D_{2k+1}, 2k+2-\nu}\in (\tilde{t}_{D_{2k+1}, 2k+2-\nu}, \tilde{t}_{D_{2k+1}, 2k+1-\nu}),\,t^{+}_{D_{2k+1}, k+2}\in(\tilde{t}_{D_{2k+1}, k+2}, 1/2).
\]

\end{proof}
\begin{lem}
\begin{equation}
\begin{array}{c}
\label{JacobiD3}
f^{+}_{D_{l}}(t) = \frac{1}{2(l-1)}(lt-2)f_{B_{l}}(t) + \frac{1}{2(l-1)}\{2(2l-1)t^2-5lt+2l\}f_{B_{l-1}}(t).
\end{array}
\end{equation}
\end{lem}
\begin{proof}
First, we recall an equality from (6.4) in \cite{[I-S]}
{\footnotesize\[
f^{+}_{D_{l}}(t) = \frac{l-2}{2l-1}f^{+}_{B_{l}}(t) + \biggl( \frac{l+1}{2l-1}- t \biggr)f^{+}_{B_{l-1}}(t).
\]}
Second, from (1.5) we have
{\footnotesize\[
f^{+}_{D_{l}}(t) = \frac{l-2}{4l-2}\Bigl( f_{B_{l}}(t) + f_{B_{l-1}}(t) \Bigr) + \frac{1}{2}\biggl( \frac{l+1}{2l-1}- t \biggr)\Bigl( f_{B_{l-1}}(t) + f_{B_{l-2}}(t)\Bigr).
\]}
Moreover, {\small$f_{B_{l}}(t)$($= (-1)^{l}\widetilde{P}_{l}^{(0, 0)}(t)$)} satisfies the $3$-term recurrence relation
{\footnotesize\[
l f_{B_{l}}(t)=(2l-1)(1-2t)f_{B_{l-1}}(t)-(l-1)f_{B_{l-2}}(t).
\]}
Hence, we obtain
{\footnotesize\[
\frac{l-2}{4l-2}\Bigl( f_{B_{l}}(t) + f_{B_{l-1}}(t) \Bigr) + \frac{1}{2}\biggl( \frac{l+1}{2l-1}- t \biggr)\Bigl( f_{B_{l-1}}(t) + f_{B_{l-2}}(t)\Bigr)
\]}
{\footnotesize\[
=\frac{1}{2(l-1)}(lt-2)f_{B_{l}}(t) + \frac{1}{2(l-1)}\{2(2l-1)t^2-5lt+2l\}f_{B_{l-1}}(t).
\]}
This coincides with the right hand side of (8.6).
\end{proof}
As a corollary of Lemma 8.10, we have.
\begin{corollary}
$t^{+}_{D_{l}, l+1-\nu} \in(t_{B_{l}, l+1-\nu}, t_{B_{l}, l-\nu})$, ($\nu=1, \ldots, l-1$), $t^{+}_{D_{l}, 1}=1$.
\end{corollary}
\begin{proof}
We note that the coefficient {\small $\frac{1}{2(l-1)}\{2(2l-1)t^2-5lt+2l\}$} in (8.6) takes positive values on the interval $[0, 1]$. Due to Lemma 8.10, the sequence $f^{+}_{D_{l}}(t_{B_{l}, l}), f^{+}_{D_{l}}(t_{B_{l}, l-1}), \ldots, f^{+}_{D_{l}}(t_{B_{l}, 1})$ has alternating sign. By applying intermediate value theorem to the polynomial {\small$f^{+}_{D_{l}}(t)$}, we obtain the results.
\end{proof}
Lastly, we prove the following theorem.
\begin{proposition}
We obtain the following inequalities for $\nu = 1, 2, \ldots, \lfloor l/2 \rfloor -1$:
\[
t^{+}_{D_{l}, l+1-\nu} < t_{D_{l}, l-\nu}.
\]

\end{proposition}
\begin{proof}
By combining Corollary 5.4 and Corollary 8.11, we obtain the results.
\end{proof}
\begin{remark} 
1. For $l=2k$, we have $1/2 < t_{D_{2k}, k}$.\\
2. For $l=2k+1$, we have $1/2 < t^{+}_{D_{2k+1}, k+1}$.
\end{remark}
\begin{proof}
1. This is an immediate consequence of Corollary 5.4.\\
2. Since $t_{B_{2k+1}, k+1}=1/2$, from Corollary 8.11, we obtain the result.
\end{proof}
\noindent
{\bf Part II}\\
We discuss location of the following roots 
\[
t_{D_{l},  \lceil l/2 \rceil+1-\nu}\,\, \mathrm{and}\,\, t^{+}_{D_{l}, \lceil l/2 \rceil+1-\nu}\,(\nu= 1, \ldots, \lceil l/2 \rceil).
\]
We recall the assumption $l \ge5$.
\begin{theorem} The following inequalities hold for $\nu= 1, \ldots, \lceil l/2 \rceil-1$
\[
t_{D_{l}, \lceil l/2 \rceil-\nu} > t^{+}_{D_{l}, \lceil l/2 \rceil+1-\nu} > t_{D_{l}, \lceil l/2 \rceil+1-\nu}.
\]

\end{theorem}
\begin{proof}
First, we prepare a proposition.
\begin{proposition}
We have the following properties:
{\small\[
t^{+}_{D_{l}, \lceil l/2 \rceil+1-\nu}\in (t_{B_{l}, \lceil l/2 \rceil+1-\nu}, t_{B_{l-1}, \lceil l/2 \rceil-\nu} ),\,(\nu=1, \ldots, \lceil l/2 \rceil-1).
\]}
\end{proposition}
\begin{proof}
First, from corollary 8.11, we obtain 
{\small\[
t^{+}_{D_{l}, \lceil l/2 \rceil+1-\nu}\in (t_{B_{l}, \lceil l/2 \rceil+1-\nu}, t_{B_{l}, \lceil l/2 \rceil-\nu}),\,(\nu=1, \ldots, \lceil l/2 \rceil-1).
\] }
Next, we note that the coefficients {\small$lt-2$} and {\small $\frac{1}{2(l-1)}\{2(2l-1)t^2-5lt+2l\}$} in (8.6) take positive values on the interval $(2/l, 1]$. Hence, by applying intermediate value theorem to the polynomial {\small$f^{+}_{D_{l}}(t)$}, we obtain the results.\\
\end{proof}
Next, from Corollary 5.4, we have
{\small\[
t_{D_{l}, \lceil l/2 \rceil+1-\nu}\in (t_{B_{l-1}, \lceil l/2 \rceil+1-\nu}, t_{B_{l}, \lceil l/2 \rceil+1-\nu}),\,(\nu=1, \ldots, \lceil l/2 \rceil).
\]}
Therefore, we have the following inequalities
{\small\[
t_{D_{l}, \lceil l/2 \rceil}<t^{+}_{D_{l}, \lceil l/2 \rceil}<\cdots<t_{D_{l}, 2}<t^{+}_{D_{l}, 2}<t_{D_{l}, 1}<1.
\]}
\end{proof}

\end{proof}
\section{Appendix I. }
From \cite{[F-R1]}\S8, we make a list of $f$-polynomials of types $A_{l}, B_{l}$ and $D_{l}$. Table A contains three infinite series $A_l \ (l\ge1),\ B_l\ (l\ge2)$ and $D_l\ (l\ge4)$. Table B contains the remaining exceptional types $E_6, E_7, E_8, F_4$ and $G_2$ and non-crystallographic types $H_3, H_4$ and $I_2(p)$. We note that in \cite{[F-Z]} Proposition 3.7 the $f$-polynomials of crystallographic types are substantially determined.\\
\bigskip
\centerline{\bf \large  Table A \quad }
%\vspace{-0.2cm}
\[ %\begin{equation}
\begin{array}{l}
f_{A_{l}}(t) \quad =\quad \sum_{k=0}^{l} (-1)^{k}\frac{1}{l+2}\binom {l}{k}\binom {l+k+2}{k+1}t^k,\,\,\,\,\,\,\,\,\,\,\,\,\,\,\,\,\,\,\,\,\,\,\,\,\,\,\,\,\,\,\,\,\,\,\,\,\,\,
\end{array}
\] %\end{equation}
\[ %\begin{equation}
\begin{array}{l}
f_{B_{l}}(t) \quad = \quad \sum_{k=0}^{l} (-1)^{k}\binom {l}{k}\binom {l+k}{k}t^k,\,\,\,\,\,\,\,\,\,\,\,\,\,\,\,\,\,\,\,\,\,\,\,\,\,\,\,\,\,\,\,\,\,\,\,\,\,\,\,\,\,\,\,\,\,\,\,\,\,\,\,\,\,
\end{array}
\] %\end{equation}
\[ %\begin{equation}
\begin{array}{l}
f_{D_{l}}(t)  \quad =\quad  \sum_{k=0}^{l} (-1)^{k} \big( \binom {l}{k}\binom {l+k-1}{k} + \binom {l-2}{k-2}\binom {l+k-2}{k}\big) t^k. 
\end{array}
\]

\bigskip
\centerline{\bf \large  Table B \quad }
%\vspace{-0.2cm}
\[ 
\begin{array}{rcl}
f_{E_{6}}(t) & \!\!\!= & 1 - 42t + 399t^2 - 1547t^3 + 2856t^4 - 2499t^5 + 833t^{6},\\
f_{E_{7}}(t) & \!\!\!= & 1 \!-\! 70t \! +\! 945t^{2}\! -\! 5180t^{3}\! +\! 14105t^{4}\! -\! 20202t^5 \! +\! 14560t^6 \! -\! 4160t^{7},\\
f_{E_{8}}(t)  & \!\!\!= & \!\!\!\!\!\!{\scriptsize 1 \!\! - \!\! 128t \!\! + \!\! 2408t^2 \!\!\! - \!\! 17936t^3 \!\!\! + \!\! 67488t^4 \!\!\! - \!\! 140448t^5 \!\!\! + \!\! 163856t^6 \!\!\! - \!\! 100320t^7 \!\!\! + \!\! 25080t^{8}}\\
f_{F_{4}}(t) & \!\!\!= & 1 - 28t + 133t^2 - 210t^3 + 105t^4,\\
f_{G_{2}}(t) & \!\!\!= & 1 - 8t + 8t^2,\\
f_{H_{3}}(t) & \!\!\!= & 1 - 18t + 48t^2 - 32t^3 ,\\
f_{H_{4}}(t) & \!\!\!= & 1 - 64t + 344t^2 - 560t^3 + 280t^4 ,\\
f_{I_2(p)}(t) &\!\!\!= & 1 - (p+2)t + (p+2)t^2 .\\
\end{array}
\]

\section{Appendix II}
 \medskip
 \noindent
 {The zero loci in the complex plane of the $f$-polynomial $f_{P}(t)$ for types $A_{20}, B_{20}, D_{20}$ and $E_8$ are exhibited in the following figures, where the zeros are indicated by $+$.} 

\bigskip
\noindent
\quad\  \includegraphics{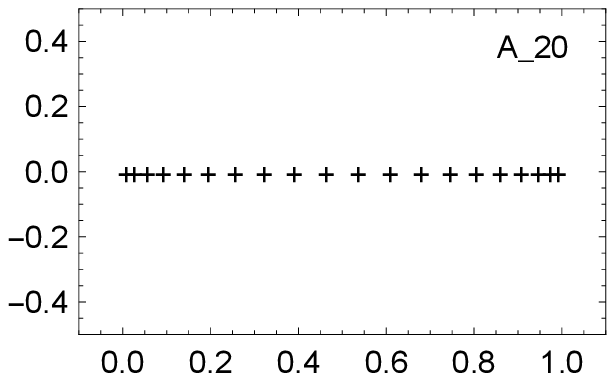}
\includegraphics{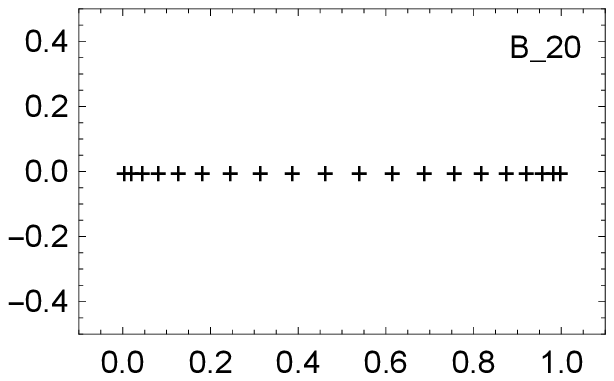}\\

\!\!\! \includegraphics{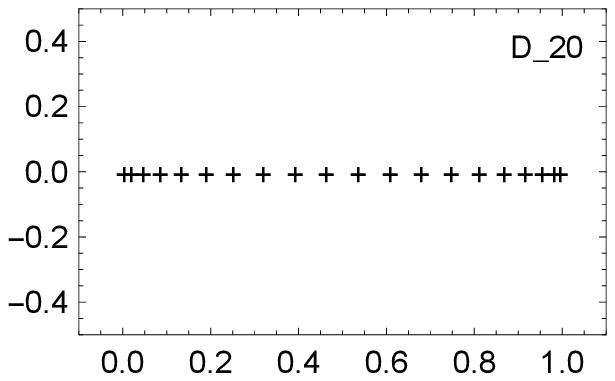}
\includegraphics{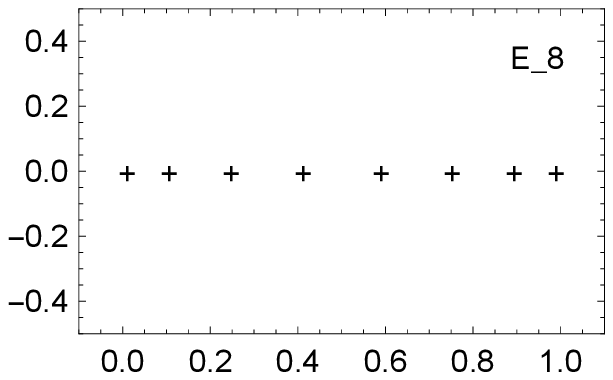}

\bigskip
\section{Addendum}
In Remark 4.4 in \cite{[I-S]}, for an Artin monoid $G^{+}_{P}$ and a dual Artin monoid $G_P^{dual+}$ of type $P$, the authors observed that the derivative at $t = 1$ of the skew-growth function {\small$N_{G_{P}^{+}}(t)$} for the Artin monoid $G^{+}_{P}$ coinsides with that of the dual Artin monoid $G_P^{dual+}$. In addition to this observation, we have found the following equalities
{\footnotesize\[
N'_{G^+_{P}}(1) = N'_{G^{dual +}_{P}}(1) = \frac{(-1)^{l}}{|W|}(lh)\prod_{i=2}^{l}(e_{i}-1),
\]}
where $h$ is the Coxeter number and $e_{1}, e_{2}, \ldots, e_{l}$ are the exponents of the corresponding finite reflection group $W$ of type $P$.\\
\bigskip
\section{Addendum}
In \cite{[I-S]}, we have studied the polynomial {\small$\widehat{N}_{G^{+}_{D_l}}(t)$ ($=f^{+}_{D_l}(t)/(1-t)$)}. We note that the polynomial {\small$\widehat{N}_{G^{+}_{D_l}}(t)$} satisfies the following Fuchsian ordinary differential equation of third-order. The proof is left to the reader.
{\footnotesize\[
 t(t-1)\{2(l-1)t-l\}\frac{ {\mathrm{d}}^{3}y }{ \mathrm{d}t^3 } + \{(l+8)(l-1)t^2-(l^2+6l-2)t+2l \}\frac{ {\mathrm{d}}^{2}y }{ \mathrm{d}t^2 }\,\,\,\,\,
\]}
{\footnotesize\[
\,\,-\{(l-1)(2l^2-5l-2)t-(l^3-2l^2-l-2)\}\frac{ \mathrm{d}y }{ \mathrm{d}t }-(l-1)^3(l+2)y=0.
\]}
\\
\emph{Acknowledgement.}~
The author was very glad to participate in the symposium hThe 3rd Franco - Japanese - Vietnamese Symposium on SingularitieshHanoi, Vietnam, November 30 - December 4, 2015. The author thanks all the organizers. The author is grateful to Kyoji Saito for enlightening discussions and his great encouragement. The author is grateful to Mutsuo Oka for his warm encouragement. This researsh was supported by World Premier International Research Center Initiative (WPI Initiative), MEXT, Japan. 
%The author expresses his deep gratitude to for their careful reading of the  ma%nuscript and for their suggestions of improvements. 

\begin{flushright}
\begin{small}
Graduate School of Mathematical Sciences, \\
University of Tokyo, Komaba, Meguro,\\
Tokyo, 153-8914, Japan \\

%Department of Mathematical Sciences, \\
%University of Tokyo, \\
%3-8-1 Komaba Meguro-ku Tokyo, 153-8914 Japan \\
e-mail address :  tishibe@ms.u-tokyo.ac.jp
\end{small}
\end{flushright}
\end{document}